\documentclass[10pt,a4paper]{amsart} 
\usepackage[english]{babel}

\usepackage[utf8]{inputenc}
\usepackage{amsmath, amssymb, amsthm, graphicx, latexsym}  
\usepackage[all]{xy} 
\usepackage{enumerate} 
\usepackage{color} 
\usepackage{hyperref}
\usepackage{fancyhdr}
\usepackage{tikz}
\usetikzlibrary{arrows}
\usepackage{multicol}

\newtheoremstyle{thm}
{}
{}
{\itshape}
{}
{\bf}
{.}
{ }
{}

\newtheoremstyle{text}
{}
{}
{\upshape}
{}
{\bf}
{.}
{ }
{}

\theoremstyle{thm}
\newtheorem{thm}{Theorem}[section]
\newtheorem{lem}[thm]{Lemma}

\theoremstyle{text}

\newtheorem{que}[thm]{Question}

\newcommand{\Z}{\mathbb{Z}}
\newcommand{\Q}{\mathbb{Q}}
\newcommand{\C}{\mathbb{C}}

\newcommand{\HP}{\mathbb{H}\mathrm{P}}

\newcommand{\CP}{\mathbb{C}\mathrm{P}}
\newcommand{\T}{\mathrm{T}}
\newcommand{\SU}{\mathrm{SU}}
\newcommand{\U}{\mathrm{U}}
\newcommand{\Sph}{\mathrm{S}}

\begin{document}   

\thispagestyle{empty}
\title[Nneg. curvature, elliptic genus and unbounded Pontryagin numbers]{Nonnegative curvature, elliptic genus and\\ unbounded Pontryagin numbers}
\author{Martin Herrmann} 
\address{Martin Herrmann \\Fakultät für Mathematik \\Karlsruher Institut für Technologie \\Kaiserstraße 89--93 \\76133 Karlsruhe, Germany}
\curraddr{Martin Herrmann \\Mathematisches Institut\\Westf\"alische Wilhelms-Universit\"at M\"unster\\Einsteinstr. 62\\48149 Münster, Germany}
\email{martin.herrmann@uni-muenster.de}

\author{Nicolas Weisskopf}
\address{Nicolas Weisskopf 
\\Département de mathématiques
\\Université de Fribourg
\\Chemin du Musée 23
\\1700 Fribourg, Switzerland}
\curraddr{Nicolas Weisskopf 
\\Court Chemin 6
\\1756 Onnens, Switzerland
}
\email{nic.weisskopf@gmail.com}
\thanks{The second author was partially supported by SNF Grant No. 200020-149761}
\begin{abstract}
We discuss the cobordism type of spin manifolds with  nonnegative sectional curvature. We show that in each dimension $4k \geq 12$, there are infinitely many cobordism types of simply connected and nonnegatively curved spin manifolds. Moreover, we raise and analyze a question about possible cobordism obstructions to nonnegative curvature.
\end{abstract}
\maketitle
\section{Introduction}
Finding obstructions to nonnegative or positive curvature on closed mani\-folds has a long tradition in Riemannian geometry. In the present article we want to deal with the question, which rational cobordism invariants can be seen as obstructions to nonnegative curvature. One such obstruction is the signature, which is bounded on connected, nonnegatively curved and oriented manifolds by Gromov's Betti number theorem. 

Normalizing the diameter and imposing an additional upper curvature bound restricts by Chern-Weil theory the Pontryagin numbers and therefore the possible oriented cobordism classes to finitely many possibilities. Nevertheless, Dessai and Tuschmann \cite{DW07} proved  that in all dimensions $4k\geq8$, there are infinitely many oriented cobordism types of simply connected and nonnegatively curved manifolds. We generalize this result to spin manifolds of nonnegative sectional curvature.

\begin{thm} \label{infinite cobordism type}
In every dimension $4k \geq 12$, there are infinitely many cobordism types of simply connected, closed, spin manifolds of nonnegative sectional curvature.
\end{thm}

Moreover, Kotschick \cite{Kot10} used slight generalizations of the examples given by Dessai and Tuschmann to show that every linear combination of Pontryagin numbers that is not a multiple of the signature is unbounded on oriented manifolds of nonnegative sectional curvature. 

We emphasize that the spin case is more difficult to treat, since there are index-theoretic obstructions. In fact, for spin manifolds with nonnegative curvature the $\hat{A}$-genus, which is the index of the Dirac operator,  vanishes by a Lichnerowicz type argument (cf. \cite{Lott00} in a more general setting). It follows that the lower bound on the dimension in Theorem \ref{infinite cobordism type} is optimal, since in dimensions 4 and 8, every Pontryagin number is a linear combination of the signature and the $\hat{A}$-genus.

Both the signature and the $\hat{A}$-genus can be seen as the first coefficient of different expansions of the elliptic genus \cite{HBJ92}. We recall that the elliptic genus $\phi(M)$ of a spin manifold $M^{4k}$ is a modular function, which takes the value of the signature in one of its cusps. In the other cusp, the elliptic genus admits the $q$-expansion
\begin{eqnarray*}
\phi(M) &=& q^{-k/2} \cdot \hat{A}(M; \bigotimes_{\substack{n \; \textrm{odd} \\ n \geq 1}}^{\infty}\Lambda_{-q^n} T_{\mathbb{C}}M \otimes \bigotimes_{\substack{n \; \textrm{even} \\ n \geq 1}}^{\infty} S_{q^n} T_{\mathbb{C}}M) \\
 &=& q^{-k/2} \cdot (\hat{A}(M) - \hat{A}(M; T_{\C} M) q \pm \ldots ) \; \in q^{-k/2} \Z[q].
\end{eqnarray*}
This expansion can be taken as a definition for the elliptic genus. 

One might ask whether the elliptic genus is constant on spin manifolds of nonnegative  sectional curvature.  For positive sectional curvature this question was raised by Dessai \cite{Des07}. To our knowledge, this question is still open.

Some evidence for a positive answer to this question is provided by the following results. First, the constancy of the elliptic genus has been shown by Hirzebruch and Slodowy \cite{HS90} in the case of homogeneous spaces. For biquotients Singhof \cite{Sin93} gave some partial results. Moreover, several results were obtained on the vanishing of the  coefficients  of the elliptic genus in the context of isometric torus actions and positive sectional curvature by Dessai  \cite{Des05}, \cite{Des07} and the second-named author \cite{Wei13}.

Since it is not evident whether the elliptic genus obstructs nonnegative curvature, we would like to raise the following question. It can be thought of as a direct analogue of Kotschick's result for spin manifolds, where we replace the signature by the elliptic genus.

 \begin{que}\label{queunbondedpontryagin}
Let $f: \Omega_{4k}^{\textrm{SO}} \otimes \Q \rightarrow \Q$ be a linear combination of Pontryagin numbers, which is not contained in the span of the coefficients of the elliptic genus. Is $f$ unbounded on connected, nonnegatively curved spin $4k$-manifolds?
\end{que} 

A positive answer to this question would imply that the elliptic genus is the only possible obstruction to nonnegative curvature on spin manifolds from the point of view of rational oriented cobordism. Here, we prove that Question \ref{queunbondedpontryagin} admits a positive answer in dimensions up to 20.
\begin{thm} \label{thmquestionsle20}
For $k\leq5$ every linear combination $f: \Omega_{4k}^{\textrm{SO}} \otimes \Q \rightarrow \Q$ of Pontryagin numbers that is not contained in the span of the coefficients of the elliptic genus is unbounded on simply connected, closed,  spin $4k$-manifolds of nonnegative sectional curvature.
\end{thm}
In dimensions 4 and 8, this theorem is trivial, since any linear combination of Pontryagin numbers in these dimensions lies in the span of the signature and the $\hat{A}$-genus and thus also in the span of the coefficients of the elliptic genus. Note that the proof of Kotschick's theorem involves the construction of a basis sequence for the rational cobordism ring consisting of nonnegatively curved manifolds. In the spin case treated here, this is not possible, since in dimension 4, any nonnegatively curved spin manifold is rationally nullbordant due to the only Pontryagin number $p_1[M^4]$ being a multiple of the $\hat{A}$-genus.

We will prove these theorems by computing the cobordism type of certain spin projective bundles over complex projective spaces.
\medskip

The paper  is structured as follows. In Section 2 we introduce the relevant families of projective bundles and discuss the cobordism type as well as the curvature properties. In Section 3 we give the proofs of the theorems.

\subsection*{Acknowledgements} These results were part of the second-named author's Ph.D. thesis \cite{Wei14}. It is a pleasure for him to thank  his advisor Anand Dessai for his support and for many helpful discussions. Most of the research was carried out in joint work during a visit of the second-named author to the Karlsruhe Institute of Technology in January 2014.

\section{Projective bundles over complex projective spaces}

For the proof of Theorem \ref{thmquestionsle20} we need to construct some nonnegatively curved families of spin manifolds, whose Pontryagin numbers are mutually distinct. For this purpose we consider the projectivization of complex vector bundles of rank $2k+2$ over the base  $\CP^{2l+1}$.  In our case, the vector bundles decompose into a sum of complex line bundles. It turns out that this construction yields a suitable family for the proof of Theorem \ref{thmquestionsle20}.

\subsection{Construction of the families}

We start with a complex vector bundle $E$ of rank $2k+2$ over $\CP^{2l+1}$. Let
\[ c(E) = 1 + c_1(E) + \ldots + c_{2k+2}(E) \in H^{\ast}(\CP^{2l+1}; \Z)\]
be the total Chern class  of the vector bundle $E$. We take the projectivization $P(E)$ with respect to $E$  and obtain a fibre bundle 
\[ \CP^{2k+1} \hookrightarrow P(E) \twoheadrightarrow \CP^{2l+1}.\]
It follows from the Leray-Hirsch Theorem  that the cohomology  ring $H^{\ast}(P(E); \Z)$ is generated as a free $H^{\ast}(\CP^{2l+1};\Z)$-module by an element $a \in H^{2}(P(E); \Z)$ subject to the following relation
\[ a^{2k+2} + a^{2k+1} c_1(E) + \ldots + c_{2k+2}(E) = 0.\]
For the notation we fix $b \in H^2(\CP^{2l+1}; \Z)$ to be a generator. 
\medskip

Next we are concerned with the spin structures of $P(E)$. We recall that a closed oriented manifold is spin if and only if its second Stiefel-Whitney class vanishes. The latter constitutes a homotopy invariant and we may compute it via the Wu formula \cite{MS74}. Alternatively, one could also apply the techniques of Borel and Hirzebruch \cite{BH58} to determine the Stiefel-Whitney classes.

\begin{lem}
$P(E)$ is spin if and only if $c_1(E)$ is even.
\end{lem}

\begin{proof}
We need to determine the second Wu class $v_2 \in H^2(P(E); \Z_2)$, which is uniquely characterized by the relation
\[ \langle v_2 \cup x, \mu_{P(E)} \rangle = \langle \mathrm{Sq}^2(x), \mu_{P(E)} \rangle,\]
where $\mu_{P(E)}$ is the fundamental class and $x \in H^{\ast}(P(E); \Z_2)$ is any \mbox{element}. The only relevant  cohomology group $H^{4k+4l+2}(P(E); \Z_2)$ is generated over $\Z_2$ by the two elements $a^{2k}b^{2l+1}$ and $a^{2k+1}b^{2l}$. We use the Cartan formula to compute the Steenrod squares
\[ \mathrm{Sq}^2 (a^{2k}b^{2l+1} ) = 0 \quad \textrm{and} \quad \mathrm{Sq}^2 (a^{2k+1}b^{2l} ) =  c_1(E) \; a^{2k+1}  b^{2l} .\]
Thus, the second Wu class vanishes if and only if $c_1(E)$ is even and so does the second Stiefel-Whitney class by the Wu formula.
\end{proof}

We recall a general recipe for the computation of the Pontryagin classes of $P(E)$. Our approach is based on the techniques of Borel and Hirzebruch. First, we observe that the fibre bundle structure of $P(E)$ induces a splitting of the tangent bundle
\begin{equation} \label{splitting}
 TP(E) \; = \;\pi^{\ast} T \CP^{2l+1} \oplus \eta_E,
\end{equation}
where $\pi^{\ast}T \CP^{2l+1}$ is the pullback bundle induced by the projection and $\eta_E$ is the complex bundle along the fibres. Following \cite[\textit{Section 15.1, p.515}]{BH58} the Chern classes of $\eta_E$  are given by
\[ c(\eta_E) = \sum_{i=0}^{2k+2} (1+a)^{2k+2-i} c_i(E).\]
From here, one can easily deduce the Pontryagin classes of $\eta_E$ and in view of the splitting (\ref{splitting}) the Pontryagin classes of $P(E)$ can be determined via the product formula.
\medskip

Following these general considerations we turn our attention to Theorem \ref{thmquestionsle20} and give explicit families for the relevant dimensions. In dimension 12 we consider complex vector bundles $E^4 \rightarrow \CP^3$ of the type 
\[ E^4 \; = \; (c \cdot \gamma^1) \oplus \epsilon^3 \quad \textrm{for} \; c \in \Z,\]
where $\gamma^1$ denotes the dual Hopf bundle, $c\cdot \gamma^1$ the $c$-fold tensor product of $ \gamma^1$ and $\epsilon^3$  is the trivial vector bundle of rank 3 over $\CP^3$. The total Chern class is then given by
\[ c(E) = 1+ c_1(E) = 1+ c \cdot b \; \in H^{\ast}(\CP^{3}; \Z).\]
As before, we take the projectivization, which we write as $X^{12}_c$. In view of the recipe, we compute the Pontryagin numbers of $X^{12}_c$. 
\begin{lem}\label{lemdimtwelve}
The Pontryagin numbers of $X^{12}_c$ are given by
\[ p_1^3[X^{12}_c] = -8c^3, \quad p_1p_2[X^{12}_c] = -6c^3, \quad p_3[X^{12}_c] = -c^3.\]
\end{lem}

In dimension 16 we take complex vector bundles $E^4 \rightarrow \CP^5$ of the type
\[ E^4 \; = \; (c \cdot \gamma^1) \oplus (2c \cdot \gamma^1) \oplus (-3c \cdot \gamma^1) \oplus \epsilon^1 \quad \textrm{for} \; c \in \Z.\]
Therefore, the total Chern class is given by
\[ c(E) = 1-7 c^2 \cdot b^2-6 c^3 \cdot b^3 \; \in H^{\ast}(\CP^5; \Z).\]
In particular, the first Chern class vanishes. As before, we projectivize these bundles to obtain a family $Y_c^{16}$ of $\CP^3$-bundles over $\CP^5$. The computation of the Pontryagin numbers is carried out according to the recipe.

\begin{lem}\label{lemdimsixteen}
The Pontryagin numbers of $Y^{16}_c$ are given by
\[p_1^4[Y^{16}_c] = 768 c^3(12+56c^2), \quad p_1^2p_2[Y^{16}_c]=384c^3(15+56c^2),\]
\[p_1p_3[Y^{16}_c] = 48c^3(42+56c^2), \quad p_2^2[Y^{16}_c]=144 c^3(24+56c^2), \quad p_4[Y^{16}_c]=288c^3.\]
\end{lem}

Our family of 20-dimensional examples is similar to the one in dimension 12. We take the rank 4 complex vector bundles
\[ E^4 \; = \; (c \cdot \gamma^1) \oplus \epsilon^3 \quad \textrm{for} \; c \in \Z\]
over $\CP^7$. We denote the projectivizations as $Z^{20}_c$.
\begin{lem}\label{lemdim20}
The Pontryagin numbers of $Z^{20}_c$ are given by
\begin{align*} 
p_1^5[Z^{20}_c] &=-64 c^3 (3 c^4+30 c^2+80), & p_1^3p_2[Z^{20}_c]&=-2 c^3 (39 c^4+480 c^2+1456), \\
p_1^2p_3[Z^{20}_c]&=-3 c^3 (3 c^4+80 c^2+352), &p_1p_2^2[Z^{20}_c]&=-c^3 (27 c^4+456 c^2+1616), \\
 p_1p_4[Z^{20}_c]&=-8 c^3 (3 c^2+29),  & p_2p_3[Z^{20}_c]&=-c^3 (3 c^4+96 c^2+580),\\
p_5[Z^{20}_c]&=-28 c^3.&& \end{align*}
\end{lem}

Finally, it is important to note that the elliptic genus vanishes on our families $X^{12}_c, \; Y^{16}_c$ and $Z^{20}_c$. This follows from a result by Ochanine \cite{Och88}  stating that the elliptic genus is multiplicative in spin fibre bundles.

\subsection{Nonnegative curvature on the families}
Next we show that the examples we will use to prove the theorems admit a metric of nonnegative curvature. We do so in a slightly more general setting.

Let $E$ be a complex vector bundle of rank $k$ over the base $B=\CP^l$ and assume it decomposes as a Whitney sum $E=\bigoplus_{i=1}^k \gamma_i$ of complex line bundles $\gamma_i$. Let $P$ be a principal $\T^k$-bundle such that $E=P\times_{\T^k} \C^k=(P\times \C^k)/\T^k$. Using a theorem of Stewart \cite{Stewart61} it is possible to lift  the standard action of $\SU(l+1)$ on $\CP^l$ to $P$ and we get that $P$ is a homogeneous space $P=(\SU(l+1)\times \T^k)/\rho(\U(l))$ where the first component $\rho_1$ of $\rho$ is a standard embedding such that $\CP^l=\SU(l+1)/\rho_1(\U(l))$ and $\rho_2$ depends on the principal torus bundle. 

Since $\rho_2|_{\SU(l)}$ has to be trivial, we get that \[\hat{P}=\SU(l+1)/\rho(\SU(l))\cong\Sph^{2l+1}\times \T^k.\] 
Let $\hat{E}=\hat{P}\times_{\T^k}\C^k$.  For the associated sphere bundles we get $ \Sph(\hat{E})=\hat{P}\times_{\T^k} \Sph^{2k-1} \cong \Sph^{2l+1}\times \Sph^{2k-1}$ and $\Sph(E)=\Sph(\hat{E})/\Sph^1$. So the projectivized bundle $\mathrm{P}(E)$ is a quotient of $\Sph^{2l+1}\times \Sph^{2k-1}$ by a free, isometric $\T^2$-action and therefore carries a metric of nonnegative curvature.

\section{Proofs of the theorems}

We combine the topological and geometric ingredients to prove Theorem \ref{thmquestionsle20}. It follows from the modular properties of the elliptic genus that its coefficients span a $(k+1)$-dimensional subspace of the dual space of $\Omega_{8k}^{\textrm{SO}} \otimes \Q$ and $\Omega_{8k+4}^{\textrm{SO}} \otimes \Q$, respectively. For the proof it suffices to show that the remaining linear combinations of Pontryagin numbers are unbounded on our families.

\begin{proof}[Proof of Theorem \ref{thmquestionsle20} for $k=3$]
Let $f: \Omega_{12}^{\textrm{SO}} \otimes \Q \rightarrow \Q$ be a linear combination of Pontryagin numbers that is not contained in the span of the coefficients of the elliptic genus. From Thom's work, it is well-known that $\Omega_{12}^{\textrm{SO}} \otimes \Q$ is a three-dimensional $\Q$-vector space.  Using \[\textrm{sign}(M)=\frac{1}{945}\big(62p_3[M]-13p_1p_2[M]+2p_1^3[M]\big)\] and \[\hat{A}(M)=\frac{1}{967680}\big(-16p_3[M]+44p_1p_2[M]-31p_1^3[M]\big)\] it is simple to check that the linear combination $f$ is given by
\[ f([M]) = \lambda_1 \textrm{sign}(M) + \lambda_2 \hat{A}_3(M) + \lambda_3 p_3[M] \quad \textrm{for} \; \lambda_i \in \Q.\]

Moreover, the coefficients of the elliptic genus in dimension 12 are spanned by the signature and the $\hat{A}$-genus and so we conclude that $\lambda_3 \neq 0$.

We now evaluate the linear combination $f$ on the family $X^{12}_{c}$. We recall that the latter is nonnegatively curved and spin for $c$ even.  We also note that by mutiplicativity, the elliptic genus vanishes on $X^{12}_{2c}$. By Lemma \ref{lemdimtwelve}, we conclude that
\[ f([X^{12}_{2c}]) = \lambda_3 p_3[X^{12}_{2c}] = - 8 \lambda_3 c^3\]
and $f$ is unbounded. Thus, we showed that $f$ is unbounded on nonnegatively curved spin manifolds, which is exactly the claim.
\end{proof}

\begin{proof}[Proof of Theorem \ref{thmquestionsle20} for $k=4$]
The proof goes along the same lines as the previous case. In this dimension the coefficients of the elliptic genus are spanned by the signature, the $\hat{A}$-genus and the index of the twisted Dirac operator $\hat{A}(M; T_{\C}M)$. Moreover, $\Omega_{16}^{\textrm{SO}} \otimes \Q$ is a five-dimensional $\Q$-vector space and $f$ is expressed by a linear combination
\[ f([M])= \lambda_1 \textrm{sign}(M) + \lambda_2 \hat{A}_4(M) + \lambda_3 \hat{A}_4 (M; T_{\C} M) + \lambda_4 p_1^4[M]+ \lambda_5 p_4[M]\]
for some $\lambda_i \in \Q$. By the multiplicativity in spin fibre bundles, the elliptic genus vanishes on $Y^{16}_c$ and in view of Lemma \ref{lemdimsixteen}, the evaluation yields
\[ f([Y^{16}_c]) = c^3 (768(12+56c^2)\lambda_4  + 288 \lambda_5).\]
Clearly, this is unbounded for $(\lambda_4, \lambda_5) \neq (0,0)$ and the claim follows.
\end{proof}

\begin{proof}[Proof of Theorem \ref{thmquestionsle20} for $k=5$]
In this dimension a linear combination of Pontryagin numbers $f: \Omega_{20}^{\textrm{SO}} \otimes \Q \rightarrow \Q$ has the form 
\begin{align*} f([M])&= \lambda_1 \textrm{sign}(M) + \lambda_2 \hat{A}_5(M) + \lambda_3  \hat{A}_5(M;T_{\C}M)+ \lambda_4 p_5[M] \\& \qquad+ \lambda_5 p_1p_4[M]+\lambda_6 p_2 p_3[M] + \lambda_7 p_1^2 p_3[M]
\end{align*}
with coefficients $\lambda_i\in  \Q$. Evaluating on $Z^{20}_{2c}$, we obtain by Lemma \ref{lemdim20}
\begin{align*}f([Z^{20}_{2c}])&=2^7 \cdot c^7 (-3 \lambda _6-9 \lambda _7)+ 2^5 \cdot c^5 (-24 \lambda _5-96 \lambda _6-240 \lambda _7)\\&\qquad+ 2^3 \cdot c^3 (-28 \lambda _4-232 \lambda _5-580 \lambda _6-1056 \lambda _7).\end{align*}
By the product formula for the Pontryagin classes we compute 
\begin{align*}f([X^{12}_{2c} \times \HP^2])&=2^3 \cdot c^3 (-7 \lambda _4-46 \lambda _5-73 \lambda _6-108 \lambda _7).\end{align*}
Both families are spin and nonnegatively curved. If $f([Z^{20}_{2c}])$ and $f([X^{12}_{2c} \times \HP^2])$ are bounded, then all coefficients of these polynomials in $c$ must vanish, which implies $\lambda_4=\lambda_5=\lambda_6=\lambda_7=0$. So a bounded linear combination $f$ has to be contained in the span of the coefficients of the elliptic genus.\end{proof}

As a consequence, we are able to prove Theorem \ref{infinite cobordism type}. The proof follows from the polynomial structure of the rational oriented cobordism ring.

\begin{proof}[Proof of Theorem \ref{infinite cobordism type}]
For a given dimension, we need to find a family of nonnegatively curved spin manifolds with mutually distinct cobordism type. We recall that the rational cobordism type is uniquely determined by the Pontryagin numbers. In dimensions 12 and 16 such families are given by $X^{12}_{2c}$ and $Y^{16}_c$ and the claim follows from Theorem \ref{thmquestionsle20}.

For the dimensions $4k \geq 20$ the family $X^{12}_{2c} \times \HP^{k-3}$ has the desired properties. In fact, this family is spin and nonnegatively curved via the product metric. On the other hand, it is  well-known that
\[ \Omega_{\ast}^{\textrm{SO}} \otimes \Q = \Q[[X_2(4)], [\HP^2], [\HP^3], \ldots].\]
In other words, the $K3$-surface $X_2(4)$ and the quaternionic projective spaces form a sequence generating the rational cobordism ring as a polynomial ring. Since $\Omega_{\ast}^{\textrm{SO}} \otimes \Q$ has the structure of a polynomial ring, multiplication with the element $[\HP^{k-3}]$ is injective for $k \geq 5$. Hence, the cobordism types of $X^{12}_{2c} \times \HP^{k-3}$ are mutually distinct, which completes the proof.
\end{proof}


\bibliographystyle{alpha}

\bibliography{references.bib}

\end{document}